\documentclass[a4paper,12pt,reqno]{amsart}
\usepackage{latexsym}
\usepackage{amssymb}
\usepackage{amsmath,color}
\usepackage{layout}

\newtheorem{teo}{Theorem}[section]
\newtheorem{lema}[teo]{Lemma}

\newtheorem{prop}[teo]{Proposition}

\newtheorem{hyp}[teo]{Hypothesis}

\newcommand{\R}{\mathbb{R}}
\newcommand{\N}{\mathbb{N}}

\newcommand{\der}{\delta}

\newcommand{\iou}{\int_{0}^{1}}
\newcommand{\iot}{\int_{0}^{t}}

\newcommand{\ist}{\int_{s}^{t}}
\newcommand{\norm}[1]{\lVert #1\rVert}

\newcommand{\ott}{[0,T]}
\newcommand{\ou}{[0,1]}

\newcommand{\bd}{\mathbf{B}^{\mathbf{2}}}

\newcommand{\bdst}{{\bf B}_{st}^{\bf 2}}

\newcommand{\cac}{\mathcal C}

\newcommand{\cl}{\mathcal L}

\newcommand{\al}{\alpha}
\newcommand{\ep}{\varepsilon}
\newcommand{\e}{\varepsilon}

\newcommand{\ga}{\gamma}

\newcommand{\si}{\sigma}

\newcommand{\lp}{\left(}
\newcommand{\rp}{\right)}

\newcommand{\lcl}{\left\{}
\newcommand{\rcl}{\right\}}

\begin{document}

\title[Weak approximation of fractional SDEs]{Weak approximation of fractional SDEs: \\  the Donsker setting}
\author[X. Bardina  \and C. Rovira \and S. Tindel
]{X. Bardina  \and C. Rovira  \and  S. Tindel}
\date{\today}

\begin{abstract}
In this note, we take up the study of weak convergence for stochastic
differential equations driven by a (Liouville) fractional Brownian motion $B$ with Hurst
parameter $H\in(1/3,1/2)$, initiated in \cite{BNRT}. In the current paper, we approximate the $d$-dimensional fBm by the convolution of a rescaled random walk with Liouville's kernel. We then show that the corresponding differential equation converges in law to a fractional SDE driven by $B$.
\end{abstract}



\thanks{This work was partially supported by MEC-FEDER Grants MTM2006-06427 and MTM2006-01351.}

\address{
{\it Xavier Bardina:}
{\rm Departament de Matem\`atiques, Facultat de Ci\`encies, Edifici C, Universitat Aut\`onoma de Barcelona, 08193 Bellaterra, Spain}.
{\it Email: }{\tt Xavier.Bardina@uab.cat}
\newline
$\mbox{ }$\hspace{0.1cm}
{\it Carles Rovira:}
{\rm Facultat de Matem\`atiques, Universitat de Barcelona, Gran Via 585, 08007 Barcelona, Spain}.
{\it Email: }{\tt carles.rovira@ub.edu}
\newline
$\mbox{ }$\hspace{0.1cm}
{\it Samy Tindel:}
{\rm Institut Élie Cartan Nancy, B.P. 239,
54506 Vanduvre-l\`es-Nancy Cedex, France}.
{\it Email: }{\tt tindel@iecn.u-nancy.fr}
}

\maketitle

\section{Introduction}

The current article can be seen as a companion paper to \cite{BNRT}, to which we refer for a further introduction. Indeed, in the latter reference, the following equation on the interval $\ou$  was considered (the generalization to $\ott$ being a matter of trivial considerations):
\begin{equation}\label{eq:eds-intro}
dy_t=\si\lp y_{t} \rp dB_t
+ b\lp   y_{t} \rp dt,\quad y_0=a\in\R^n,
\end{equation}
where $\sigma:\R^n\rightarrow \R^{n\times d}$, $b:\R^n\rightarrow\R^n$ are two bounded  and smooth enough functions, and $B$ stands for a $d$-dimensional fBm with Hurst parameter $H>1/3$.

\smallskip

Let us be more specific about the driving process for equation (\ref{eq:eds-intro}): we consider in the sequel the so-called $d$-dimensional Liouville fBm $B$, with Hurst parameter $H\in(1/3,1/2)$. Namely, $B$ can be written as $B=(B^1,\ldots,B^d)$, where the $B^i$'s are $d$ independent centered Gaussian processes of the form
\begin{equation}\label{eq:liouville-intro}
B_t^i=\iot (t-r)^{H-\frac12} dW_r^i,
\end{equation}
for a $d$-dimensional Wiener process $W=(W^1,\ldots,W^d)$. This process is very close to the usual fBm, in the sense that they only differ by a finite variation process (as pointed out in \cite{AMN}), and we shall see that its simple expression (\ref{eq:liouville-intro}) simplifies some of the computations in the sequel. In any case, $B$ falls into the scope of application of the rough paths theory, which means that equation (\ref{eq:eds-intro}) can be solved thanks to the semi-pathwise techniques contained in \cite{FVbk,Gu,LyonsBook}. The natural question raised in \cite{BNRT} was then the following: is it possible to approximate equations like (\ref{eq:eds-intro}) in law by ordinary differential equations, thanks to a Wong-Zakai type approximation (see \cite{KP,TZ,WZ} for further references on the topic)?

\smallskip

Some positive answer to this question had already been given in \cite{FV}, where some Gaussian sequences approximations were considered in a general context. In \cite{BNRT}, we focused on a natural and easily implementable (non Gaussian) scheme for $B$, based on Kac-Stroock's approximation to white noise (see \cite{kac, stroock}). However, another very natural way to approximate $B$ relies on Donsker's type scheme (see \cite{So} for the case $H>1/2$ and \cite{BFH} for the Brownian case), involving a rescaled random walk. We have thus decided to investigate weak approximations to (\ref{eq:eds-intro}) based on this process.

\smallskip

More precisely, as an approximating sequence of $B$, we shall choose $(X^{\e})_{\e>0}$, where
$X^{\e,i}$ is defined
as follows for $i=1,\ldots,d$: consider a family of independent random variables $\{\eta_k^i; \, k\ge 1, \, 1\le i\le d\}$, satisfying the
\begin{hyp}\label{hyp:1}
The random variables $\{\eta_k^i; \, k\ge 1, \, 1\le i\le d\}$ are independent and share the same law as another random variable $\eta$. Furthermore, $\eta$ is assumed to satisfy $E\left(\eta\right)=0, E\left(\eta^2\right)=1$ and is almost surely bounded by a constant $k_\eta$.
\end{hyp}
We then define $X^{\e,i}$ in the following way:
\begin{equation}\label{eq:def-x-ep}
X^{i,\e}(t)=\int_0^t (t+\e-r)^{H-\frac12}\theta^{i,\e}(r)dr,
\end{equation}
where
\begin{equation}\label{donsker}
\theta^{i,\varepsilon}(r):=\frac1{\varepsilon}\sum_{k=1}^{+\infty}\eta_k^i \, I_{[k-1,k)}
\left(\frac{r}{\varepsilon^2}\right).
\end{equation}
Notice that $X^\ep$ is really a process given by the convolution of the rescaled random walk $\theta^{\varepsilon}$ with Liouville's kernel.

\smallskip

Let us then consider the process $y^\e$ solution to equation (\ref{eq:eds-intro})
driven by $X^\e$, namely:
\begin{equation}\label{eq:eds-approx}
dy_t^\e=\si\lp y_{t}^\e \rp dX_t^\e
+ b\lp   y_{t}^\e \rp dt,\quad y_0^\e=a\in\R^n, \quad t\in\ott.
\end{equation}
Our main result is as follows:
\begin{teo}\label{thm:approx-diffusion}
Let $(y^\e)_{\e>0}$ be the family of processes defined by (\ref{eq:eds-approx}),
and let $1/3<\gamma<H$, where $H$ is the Hurst parameter of $B$.
Then, as $\ep\to 0$, $y^\ep$ converges in law to the process $y$ obtained as
the solution to (\ref{eq:eds-intro}), where the convergence takes place in the
H\"older space $C^\ga([0,1];\R^n)$.
\end{teo}
This theorem is obtained invoking many of the techniques introduced in \cite{BNRT}. In the end, as explained at Section \ref{sec:reduc-problem}, most of the technical differences between the two articles arise in the way to evaluate the moments of quantities like $\iou f(r) \theta^{i,\varepsilon}(r) \, dr$ for a given H\"older function $f$, and to compare them with the moments of  Gaussian random variable. This is where we shall concentrate our efforts in the remainder of the paper, mostly at Section~\ref{sec:moments-estimates}.

\section{Reduction of the problem}\label{sec:reduc-problem}

We shall recall here briefly some preliminary steps contained in \cite{BNRT}, which allow to reduce our problem to the evaluation of the moments of a specific type of Wiener integrals.

\smallskip

First of all, we need to recall the definition of some H\"older spaces, in which our convergences take place. We call for instance $\cac_j(\ou;\R^d)$ the space of continuous functions from $\ou^j$ to $\R^d$, which will mainly be considered for $j=1$ or $2$ variables. The H\"older norms on those spaces are then defined in the following way: for $f \in \cac_2(\ou;\R^d)$ let
$$
\norm{f}_{\mu} =
\sup_{s,t\in\ott}\frac{|f_{st}|}{|t-s|^\mu},
\quad\mbox{and}\quad
\cac_2^\mu(\ou;\R^d)=\lcl f \in \cac_2(\ou;\R^d);\, \norm{f}_{\mu}<\infty  \rcl.
$$
The usual H\"older spaces $\cac_1^\mu(\ou;\R^d)$  are then determined by setting $\|g\|_{\mu}=\|\der g\|_{\mu}$ for a continuous function $g\in\cac_1(\ou;\R^d)$, where $\der g\in\cac_2(\ou;\R^d)$ is defined by $\der g_{st}=g_t-g_s$. We then say that $g\in\cac_1^\mu(\ou;\R^d)$ iff $\|g\|_{\mu}$ is finite.
Note that $\|\cdot\|_{\mu}$ is only a semi-norm on $\cac_1(\ou;\R^d)$,
but we will  work  in general on spaces of the type
\begin{equation}\label{def:hold-init}
\cac_{1,a}^\mu(\ou;\R^d)=
\lcl g:\ott\to V;\, g_0=a,\, \|g\|_{\mu}<\infty \rcl,
\end{equation}
for a given $a\in V$, on which $\|g\|_{\mu}$  is a norm.

\smallskip

The second crucial point one has to recall is the natural definition of a L\'evy area for Liouville's fBm:
\begin{prop}\label{prop:hyp-fbm}
Let $B$ be a $d$-dimensional Liouville fBm,
and suppose that its Hurst parameter satisfies $H\in(1/3,1/2)$. Then

\smallskip

\noindent
{\bf (1)} $B$ is almost surely a $\ga$-H\"older path for any $1/3<\ga<H$.

\smallskip

\noindent
{\bf (2)} A L\'evy area based on $B$ can be defined by setting
$$
\bdst=\ist dB_u \otimes \int_s^u dB_v,
\quad\mbox{i. e.}\quad
\bdst(i,j)=\ist dB_u^i  \int_s^u dB_v^j,
\quad i,j\in\{1,\ldots,d  \},
$$ for $0\le s < t \leq T$. Here,
the stochastic integrals are defined as Wiener-It\^o integrals when $i\neq j$, while, when $i=j$, they are
simply given by
$$
\ist dB_u^i  \int_s^u dB_v^i = \frac12\left(B_t^i - B_s^i\right)^2.
$$

\smallskip

\noindent
{\bf (3)} The process $\bd$ is almost surely an element of $\cac_2^{2\ga}(\ou;\R^{d\times d})$, and satisfies the algebraic relation
$$
\bd_{st}-\bd_{su}-\bd_{ut}= \lp  B_u-B_s \rp \otimes \lp  B_t-B_u \rp,
$$
for all $0\le s \le u \le t\le 1$.
\end{prop}

\smallskip

These algebraic and analytic properties of the fBm path allow to invoke the rough path machinery (see \cite{FVbk,Gu,LyonsBook}) in order to solve equation (\ref{eq:eds-intro}):

\begin{teo}\label{thm:ex-uniq1}
Let $B$ be a Liouville fBm with Hurst parameter $1/3<H<1/2$, and
$\si:\R^n\to\R^{n\times d}$ be a $C^2$ function, which is bounded
together with its derivatives. Then

\smallskip

\noindent
{\bf (1)}
Equation (\ref{eq:eds-intro}) admits a unique solution $y\in\cac_1^{\ga}(\R^n)$ for any $1/3<\ga<H$, with the additional structure of weakly controlled process introduced in \cite{Gu}.

\smallskip

\noindent
{\bf (2)}
The mapping $(a,B,\bd)\mapsto y$ is continuous from
$\R^n\times\cac_1^{\ga}(\R^d)\times\cac_2^{2\ga}(\R^{d\times d})$
to $\cac_1^{\ga}(\R^n)$.
\end{teo}

\smallskip

One of the nice aspects of rough paths theory is precisely the second point in Theorem~\ref{thm:ex-uniq1}, which allows to reduce immediately our weak convergence result for equation~(\ref{eq:eds-intro}), namely Theorem \ref{thm:approx-diffusion}, to the following result on the approximation of $(B,\bd)$:

\begin{teo}\label{main-thm-weak}
Recall that the random variables $\eta_k^i$ satisfy Hypothesis \ref{hyp:1}, and let $X^\ep$ be defined by (\ref{eq:def-x-ep}). For any $\e>0$, let ${\bf X}^{\bf 2,\e}=({\bf X}^{\bf 2,\e}_{st}(i,j))_{s,t\ge 0;\,i,j=1,\ldots,d}$
be the natural L\'evy's area associated to $X^\e$,
given by
\begin{equation}\label{levydef}
{\bf X}^{\bf 2,\e}_{st}(i,j)=\int_s^t (X^{j,\e}_u-X^{j,\e}_s) \, dX^{i,\e}_u,
\end{equation}
where the integral is understood in the usual Lebesgue-Stieltjes sense.
Then, as $\e\rightarrow 0$,
\begin{equation}\label{loi}
(X^\e,{\bf X}^{\bf 2,\e})
\,\,\,{\stackrel{{\rm Law}}{\longrightarrow}}\,\,\,
(B,{\bf B}^{\bf 2}),
\end{equation}
where ${\bf B}^{\bf 2}$ denotes the L\'evy area defined in Proposition \ref{prop:hyp-fbm},
and where the convergence in law holds in the spaces
$\mathcal{C}_1^\mu(\R^d)\times\mathcal{C}_2^{2\mu}(\R^{d\times d})$,
for any $\mu<H$.
\end{teo}
The remainder of our work is thus devoted to the proof of Theorem \ref{main-thm-weak}.

\smallskip

As usual in the context of weak convergence of stochastic processes, we divide the proof into
weak convergence for finite-dimensional distributions  and
a tightness type result. Furthermore, the tightness result in our case is easily deduced from the analogous result in \cite{BNRT}:
\begin{prop}
The sequence $(X^\e,{\bf X}^{\bf 2,\e})_{\ep>0}$ defined at Theorem \ref{main-thm-weak} is tight in
$\mathcal{C}_1^\mu(\R^d)\times\mathcal{C}_2^{2\mu}(\R^{d\times d})$.
\end{prop}

\begin{proof}
The proof follows exactly the steps of \cite[Proposition 4.3]{BNRT}, the only difference being that our Lemma \ref{lema01} has to be applied here in order to get the equivalent of inequality~(28) in \cite{BNRT}. Details are left to the reader.

\end{proof}

\smallskip

With these preliminaries in hand, we can now turn to the finite dimensional distribution (f.d.d. in the sequel) convergence, which can be stated as:
\begin{prop}\label{prop-weak}
Under the assumption \ref{hyp:1}, let $(X^\e,{\bf X}^{\bf 2,\e})$ be the approximation process defined by
(\ref{eq:def-x-ep}) and (\ref{levydef}). Then
\begin{equation}\label{fdd}
{\rm f.d.d.}-\lim_{\e\to 0} (X^\e,{\bf X}^{\bf 2,\e})
= (B,{\bf B}^{\bf 2}),
\end{equation}
where ${\rm f.d.d.}-\lim$ stands for the convergence in law of the
finite dimensional distributions. Otherwise stated, for any $k\ge 1$ and any family
$\{s_i,t_i;\, i\le k, 0\le s_i<t_i\le T\}$, we have
\begin{equation}\label{fddbis}
\cl-\lim_{\e\to 0}
(X_{t_1}^\e,{\bf X}_{s_1 t_1}^{\bf 2,\e},\ldots, X_{t_k}^\e,{\bf X}_{s_k t_k}^{\bf 2,\e})
= (B_{t_1},{\bf B}_{s_1 t_1}^{\bf 2},\ldots, B_{t_k},{\bf B}_{s_k t_k}^{\bf 2}).
\end{equation}
\end{prop}

\begin{proof}
The structure of the proof follows again closely the steps of \cite[Proposition 5.1]{BNRT}, except that other kind of estimates will be needed in order to handle the Donsker case.

\smallskip

To be more specific, it should be observed that the first series of simplifications in the proof of \cite[Proposition 5.1]{BNRT} can be repeated here. They allow to pass from a convergence of double iterated integrals to the convergence of some Wiener type integrals with respect to $X^\ep$. Namely, for $i=1,2$ and $0\le u<t\le 1$, set
$$
Y^i(u,t)=\int_u^t (B^i_v - B^i_u)(v-u)^{H-\frac32}dv,
$$
and for $0\le u<t\le 1$ and $(u_1,\ldots,u_6)$ in a neighborhood of 0 in $\R^6$, set also
\begin{eqnarray*}
Z_u&=&u_1 + u_2 B^2_u + u_3 Y^{2}(u,t) + u_4 \int_u^t (v-u)^{H-\frac12}dW^2_v\\
&&+u_5\int_u^t dw \int_u^w (w-v)^{H-\frac32}\big(
(w-u)^{H-\frac12}-(v-u)^{\frac12}
\big) \, dW^2_v\\
&&+u_6\int_u^t dw \int_0^u (w-v)^{H-\frac32}
(w-u)^{H-\frac12} \, dW^2_v.
\end{eqnarray*}
Consider the analogous processes $Y^{i,\ep},Z^{\ep}$ defined by the same formulae, except that they are based on the approximations $\theta^{i,\ep}$ of white noise. We still need to recall a little more notation from \cite{BNRT}: for $f\in L^2([0,1])$ and $t\in\ou$, we set
\begin{equation}\label{eq:def-Phi-ep-phi-f}
\Phi_\e(f)=E\left( e^{i\int_0^t f(u)\theta^{\e,1}(u)du} \right),
\quad\mbox{ }\quad
\phi_f^{\varepsilon}=\int_0^1\int_0^1f^2(x)f^2(y)I_{\{|x-y|<\varepsilon^2\}}dxdy,
\end{equation}
and
$$
\Phi(f)=E\left( e^{i\int_0^t f(u) \, dW^{1}_u} \right)=e^{\frac12 \int_0^t f^2(u) \, du}.
$$
Then it is shown in \cite[Proposition 5.1]{BNRT} that one is reduced to prove that $\lim_{\e \to 0} v_\e^a = 0$, where $v_\e^a$ is given by
$$
v_\e^a =E\left( \Phi_\e(Z^\e)e^{i w \int_0^t \theta^{\e,2}(u)du}\right)  -
E\left( \Phi(Z^\e)e^{i w \int_0^t \theta^{\e,2}(u)du}\right),
$$
for an arbitrary real parameter $w$ in a neighborhood of 0. Furthermore, bounding $e^{i w \int_0^t \theta^{\e,2}(u)du}$ trivially by 1 and conditioning, it is easily shown that $v_\e^a$ is controlled by the difference $E[ |\Phi_\e(Z^\e))  -\Phi(Z^\e)|]$, for which Lemma~\ref{lm-control2} provides the bound
\begin{align*}
&| \Phi_\e(Z^\e))  - \Phi(Z^\e)| \\
 & \leq E  \left[
4\sqrt{\frac15}w^3(\phi_{Z_{\e}}^{\varepsilon})^{\frac12}\|Z_{\e}\|_{L^2}
k_\eta^3\exp(4 w^2 k_\eta^2 \|Z_{\e}\|_{L^2}^2) +w^2\varepsilon^{2\alpha}\|Z_{\e}\|_{\alpha}\|Z_{\e}\|_{L^2}\exp(w^2\|Z_{\e}\|_{L^2}^2) \right.\\
&\qquad + \left.
\frac8{\sqrt{5}}w^4(\phi_{Z_{\e}}^{\varepsilon})^{\frac12}\|Z_{\e}\|_{L^2}^2
{k_\eta}^2\exp(4w^2 k_\eta^2 \|Z_{\e}\|_{L^2}^2)+\frac12
w^4\phi_{Z_{\e}}^{\varepsilon}\exp(w^2\|Z_{\e}\|_{L^2}^2)\right],
\end{align*}
for any $\alpha\in(0,1)$. In order to reach our aim, it is thus sufficient to check the following inequalities:
$$
\sup_{\e} E \left( \|Z_{\e}\|_\alpha^2 \right)\le M, \quad \lim_{\e
\to 0} E \left( (\phi_{Z_{\e}}^{\varepsilon})^2 \right) =0
$$
and for $w<w_0$, where $w_0$ is a small enough constant,
$$
\sup_{\e} E \left(  {\rm e}^{{w^2  \|Z_{\e}\|^2_{L^2}}}\right) \le
M.
$$
However, these relations can be deduced, as (39), (40) and (41) in
\cite{BNRT}, from Lemma \ref{lema01} (it should be noticed however that a
one-parameter version of  \cite[Lemma 5.1]{BJQ} is needed for the adaptation of the latter result to our Donsker setting). The proof is thus
finished once the lemmas below are proven.

\end{proof}

\section{Moments estimates in the Donsker setting}\label{sec:moments-estimates}
In order to deal with our technical estimates, let us first introduce a new notation: set $\rho_1=(1-5^{1/2})/2$ and $\rho_2=(1+5^{1/2})/2$. Then the moments of any integral of a deterministic kernel $f$ with respect to $\theta^{i,\e}$ can be bounded as follows:
\begin{lema}\label{lema01}
Let $m\in\N$,  $f\in L^2([0,1])$, $i \in \{1,2\}$ and $\varepsilon >0$. Recall that the random variable $\eta$ is assumed to be almost surely bounded by a constant $k_\eta$. Then we have
\begin{align}\label{des2m}
&\left|E\left[\left(\int_0^1f(r)\theta^{i,\varepsilon}(r)dr\right)^{2m}\right]\right| \\
&\leq \frac{(2m)!}{2^m \, m!} \|f\|_{L^2}^{2m} \, +
\frac{(2m)!}{5^{1/2}(m-2)!}k_\eta^{2m} \,
{\left(\rho_2^{2m-1}-\rho_1^{2m-1}\right)}(\phi_f^{\varepsilon})^{\frac12}
\|f\|_{L^2}^{2m-2}, \notag
\end{align}
and
\begin{equation}\label{des2mmu}
\left|E\left[\left(\int_0^1f(r)\theta^{i,\varepsilon}(r)dr\right)^{2m+1}\right]\right|
\leq \frac{(2m+1)!}{5^{1/2}(m-1)!}k_\eta^{2m+1} \,
\left(\rho_2^{2m}-\rho_1^{2m}\right)
(\phi_f^{\varepsilon})^{\frac12}\|f\|_{L^2}^{2m-1},
\end{equation}
where $\phi_f^{\varepsilon}$ is the quantity defined at (\ref{eq:def-Phi-ep-phi-f}).
\end{lema}

\begin{proof}
We focus first on inequality (\ref{des2m}) and divide this proof into several steps.

\smallskip

\noindent
\textit{Step 1: Identification of some key iterated integrals.} Notice that
\begin{multline*}
\left|E\left[\left(\int_0^1f(r)\theta^{i,\varepsilon}(r)dr\right)^{2m}\right]\right|  \\
\leq
\int_{[0,1]^{2m}}|f(r_1)|\cdots|f(r_{2m})||E(\theta^{i,\varepsilon}(r_1)\cdots\theta^{i,\varepsilon}(r_{2m}))|
dr_1\cdots dr_{2m}.
\end{multline*}
Transforming the symmetric integral on $\ou^{2m}$ into an integral on the simplex, and using expression (\ref{donsker}) for $\theta^\e$, we can write the latter expression as:
\begin{multline}\label{sumatori}
\frac{(2m)!}{\varepsilon^{2m}}
\sum_{\footnotesize{\begin{array}{c}k_1,\dots,k_{2m}=1\\k_1\geq\cdots
\geq k_{2m}\end{array}}}^{n(\varepsilon)}\int_{(k_1-1)\varepsilon^2}^{k_1\varepsilon^2}\cdots
\int_{(k_{2m}-1)\varepsilon^2}^{k_{2m}\varepsilon^2}|f(r_1)|\cdots|f(r_{2m})|\\
 \times |E(\eta_{k_1}^i\cdots\eta_{k_{2m}}^i)|I_{\{r_1\geq
r_2\geq\cdots\geq r_{2m}\}}dr_1\cdots dr_{2m},
\end{multline}
where $n(\varepsilon)=[\frac{1}{\varepsilon^2}]+1$ and where we understand that $f(x)=0$ whenever $x>1$.

\smallskip

Let us study now the quantities
$E(\eta_{k_1}^i\cdots\eta_{k_{l}}^i)$. If there exists $l$ such that
$k_l\neq k_j$  for all $j\neq l$ then
$E(\eta_{k_1}^i\cdots\eta_{k_{l}}^i)=0$. On the other hand, when
$k_{2l-1}=k_{2l}> k_{2l+1}$ for any $l$, we clearly have
$E(\eta_{k_1}^i\cdots\eta_{k_{l}}^i)=1.$ Finally, in the general
case, for all $l\in\N$, $ |E(\eta_{k_1}^i\cdots\eta_{k_{l}}^i)|\le
k_\eta^l. $ Separating the cases in this way for
$|E(\eta_{k_1}^i\cdots\eta_{k_{2m}}^i)|$, we end up with a
decomposition of the form
$E[(\int_0^1f(r)\theta^{i,\varepsilon}(r)dr)^{2m}]$ $= T_m^1 +
T_m^2$, where
\begin{multline*}
T_m^1 =
\frac{(2m)!}{\varepsilon^{2m}}\sum_{\footnotesize{\begin{array}{c}k_1,\dots,k_{m}=1\\k_1>\cdots>k_m\end{array}}}^{n(\varepsilon)}
\int_{(k_1-1)\varepsilon^2}^{k_1\varepsilon^2}\int_{(k_1-1)\varepsilon^2}^{k_1\varepsilon^2}\cdots
\int_{(k_{m}-1)\varepsilon^2}^{k_{m}\varepsilon^2}\int_{(k_{m}-1)\varepsilon^2}^{k_{m}\varepsilon^2}
\\
|f(r_1)|\cdots |f(r_{2m})|\times I_{\{r_1\geq r_2\geq\cdots\geq
r_{2m}\}}dr_1\cdots dr_{2m},
\end{multline*}
and where the term $T_m^2$ is defined by:
\begin{equation}\label{eq:def-T2}
T_m^2 =  \frac{(2m)!\, k_\eta^{2m}}{\varepsilon^{2m}}
\sum_{\footnotesize{\begin{array}{c}n_1,\dots,n_{s} \ge 2; s \in
\{1,...,m-1\} \\n_1+\cdots+n_s=2m\end{array}}} U_{n_1,\dots,n_{s}},
\end{equation}
with
\begin{equation}\label{eq:def-U-n1-ns}
U_{n_1,\dots,n_{s}}=
\sum_{\footnotesize{\begin{array}{c}k_1,\dots,k_{s}=1\\k_1>\cdots>k_s\end{array}}}^{n(\varepsilon)}
\int_{D_{k_1\cdots k_s}}
|f(r_1)|\cdots |f(r_{2m})| I_{\{r_1\geq r_2\geq\cdots\geq
r_{2m}\}}dr_1\cdots dr_{2m},
\end{equation}
and where we have set $D_{k_1\cdots k_s}=\prod_{j=1}^{s}[(k_j-1)\varepsilon^2, k_j\varepsilon^2]^{n_j}$.

\smallskip

Let us observe at this point that we have split our sum into $T_m^1$ and $T_m^2$ because $T_m^1$ represents the dominant contribution to our moment estimate. This is simply due to the fact that $T_m^1$ is obtained by assuming some pairwise equalities among the random variables $\eta_k^i$, while $T_m^2$ is based on a higher number of constraints. In any case, both  expressions will be analyzed through the introduction of some iterated integrals of the form
$$
K_\nu(k;v,w)= \frac{1}{\ep^{\nu}}\int_{[(k-1)\varepsilon^2,
k\varepsilon^2]^{\nu}} \prod_{j=1}^{\nu} |f(r_j)| \, I_{\{w\ge
r_1>\cdots>r_\nu\ge v\}} dr_1\cdots dr_\nu,
$$
defined for $\nu,k\ge 1$ and $0\le v<w\le 1$.

\smallskip

\noindent
\textit{Step 2: Analysis of the integrals $K_\nu$.} Those iterated integrals are treated in a slightly different way according to the parity of $\nu$. Indeed, for $\nu=2n$, thanks to the elementary inequality $2ab\leq a^2+b^2$, we obtain a bound of the form:
\begin{align}\label{cas2n}
&\sum_{k=1}^{n(\varepsilon)}K_{2n}(k;v,w) \\
&\le \sum_{k=1}^{n(\varepsilon)}\frac1{\varepsilon^{2n}}\int_{[(k-1)\ep^{2},k\ep^{2}]^{2n}}\prod_{i=1}^n\left(\frac{f^2(x_{2i})+f^2(x_{2i+1})}{2}\right)I_{\{w\ge x_1\geq \cdots\geq x_{2n}\geq v\}}dx_1\cdots dx_{2n} \notag \\
&= \sum_{k=1}^{n(\varepsilon)}\frac{1}{\varepsilon^{2n}}\int_{[(k-1)\ep^{2},k\ep^{2}]^{2n}}f^2(x_{1})\cdots f^2(x_{n})I_{\{w\ge x_1\geq \cdots\geq x_{n}\geq v\}}dx_1\cdots dx_{2n} \notag \\
&= \sum_{k=1}^{n(\varepsilon)}\int_{[(k-1)\ep^{2},k\ep^{2}]^{n}}f^2(x_{1})\cdots f^2(x_{n})I_{\{w\ge x_1\geq \cdots\geq x_{n}\geq v\}}dx_1\cdots dx_{n}  \notag\\
&\le \int_{[0,1]^n}f^2(x_1)\cdots
f^2(x_{n})I_{\{x_1-x_n<\varepsilon^2\}}I_{\{w\ge x_1\geq \cdots\geq
x_{n}\ge v\}}dx_1\cdots dx_{n}. \notag
\end{align}
The case $\nu=2n+1$ can be treated along the same lines, except for
the fact that one has to cope with some expressions of the form
\begin{eqnarray}\label{cas3}
\sum_{k=1}^{n(\varepsilon)}K_3(k;v,w)&\le&
 \sum_{k=1}^{n(\varepsilon)}\frac1{\varepsilon}\int_{[(k-1)\varepsilon^2, k\varepsilon^2]^2}
 |f(x_1)|f^2(x_2)I_{\{w\ge x_1\geq x_2\geq v\}}dx_1dx_2 \notag \\
& \le& \frac1{\varepsilon}\int_{[0,1]^2} |f(x_1)|f^2(x_2)I_{\{x_1-x_2<\varepsilon^2\}}
I_{\{w \ge x_1\geq x_2\geq v\}}dx_1dx_2.
\end{eqnarray}

\smallskip

Combining (\ref{cas3}) and (\ref{cas2n}) we can state the following general formula: let $\nu\ge 1$, and define a couple $(\nu^*,\hat \nu)$ as: (i) $\nu^*=\nu/2$, $\hat \nu=0$ if $\nu$ is even, (ii) $\nu^*=(\nu+1)/2$, $\hat \nu=1$ if $\nu$ is odd. With this notation in hand, we have:
\begin{multline}\label{sanars}
\sum_{k=1}^{n(\varepsilon)}K_{\nu}(k;v,w) \le
\frac1{\varepsilon^{\hat \nu}}\int_{[0,1]^{\nu^*}}|f(x_1)|^{2-\hat
\nu} f^2(x_2)\cdots f^2(x_{\nu^*})
I_{\{x_1-x_{\nu^*}<\varepsilon^2\}} \\
\times I_{\{w\ge
x_1\geq \cdots\geq x_{\nu^*}\ge v\}}dx_1\cdots dx_{\nu^*}.
\end{multline}

\smallskip

\noindent \textit{Step 3: Bound on $T_m^1$.} It is readily checked
that $T_m^1$ can be decomposed into blocks of the form $K_2(k;w,v)$,
for which one can apply (\ref{sanars}). This yields
\begin{align*}
& T_m^1 \leq \frac{(2m)!}{2^m}
\sum_{k_1,\dots,k_{m}=1}^{n(\varepsilon)}\int_{(k_1-1)\varepsilon^2}^{k_1\varepsilon^2}\cdots
\int_{(k_{m}-1)\varepsilon^2}^{k_{m}\varepsilon^2} f^2(r_1)\cdots
f^2(r_{m})I_{\{r_1\geq
r_2\geq\cdots\geq r_{m}\}}dr_1\cdots dr_{m}\\
& \quad \leq \frac{(2m)!}{2^m m!}\|f\|_{L^2}^{2m}.
\end{align*}

\smallskip

\noindent
\textit{Step 4: Bound on $U_{n_1,\dots,n_{s}}$.}
Recall that $U_{n_1,\dots,n_{s}}$ is defined by (\ref{eq:def-U-n1-ns}). We introduce now a recursion procedure in order to control this term. Namely, integrating with respect to the last $n_s$ variables, one obtains that
\begin{multline*}
\frac1{\varepsilon^{2m}}U_{n_1,\dots,n_{s}} =
\frac{1}{\varepsilon^{2m-n_s}}
\sum_{\footnotesize{\begin{array}{c}k_1,\dots,k_{s-1}=1\\k_1>\cdots>k_{s-1}\end{array}}}^{n(\varepsilon)}
\int_{D_{k_1\cdots k_{s-1}}} \prod_{l=1}^{2m-n_s}  |f(r_l)| \,
I_{\{r_1\geq r_2\geq\cdots\geq
r_{2m-n_s}\}}  \\
\times K_{n_s} (k_s;0;r_{2m-n_s}) \,
dr_1\cdots dr_{2m-n_s}
\end{multline*}
Plugging our bound (\ref{sanars}) on $K_{n_s}$ into this expression, we get
\begin{multline*}
\frac1{\varepsilon^{2m}}U_{n_1,\dots,n_{s}} \le
 \frac{1}{\varepsilon^{2m-n_s+\hat n_s}}
\sum_{\footnotesize{\begin{array}{c}k_1,\dots,k_{s-1}=1\\k_1>\cdots>k_{s-1}\end{array}}}^{n(\varepsilon)}
\int_{D_{k_1\cdots k_{s-1}}} \int_{[0,1]^{n_s^*}}
\prod_{l=1}^{2m-n_s}  |f(r_l)|  \ |f(y_1)|^{2-\hat n_s} \\
\times   \prod_{j=2}^{n_s^{*}} f^2(y_{j})
 I_{\{y_1-y_{n_s^*}<\varepsilon^2\}} I_{\{r_1\geq r_2\geq\cdots\geq
r_{2m-n_s}\geq
y_1\geq \cdots\geq y_{n_s^*}\}}
dr_1\cdots dr_{2m-n_s} dy_1\cdots dy_{n_s^*}.
\end{multline*}
We can now proceed, and integrate with respect to the variables
$r_l$ for $k_{s-1}\le l \le k_s$. In the end, since $\sum n_s=2m$, the remaining singularity in $\ep$ is of the form $\prod\ep^{-\hat n_s}$. However, each of the singularity $\ep^{-\hat n_s}$ comes with an integral $\phi_q^\ep$ with $q=|f|^{1/2}$. The latter integral is easily seen to be of order $\ep$, which compensates the singularity $\ep^{-\hat n_s}$ (recall that $\hat n_s\le 1$). Hence, iterating the integrations with respect to the variables $r$, we
end up with a bound of the form
\begin{equation}\label{pellema02}
\frac1{\varepsilon^{2m}}U_{n_1,\dots,n_{s}}
\leq\frac1{(m-2)!}\|f\|_{L^2}^{2(m-2)}\phi_f^{\varepsilon} \leq
\frac1{(m-2)!}\|f\|_{L^2}^{2m-2}(\phi_f^{\varepsilon})^{\frac12}.
\end{equation}

\smallskip

\noindent \textit{Step 5: Bound on $T_m^2$.} Owing to inequality
(\ref{pellema02}), our bound on $T_m^2$ can be reduced now to an
estimate of the  number of terms in the sum over $n_1,\ldots,n_s$ in
formula (\ref{eq:def-T2}). This boils down to the following
question: given a natural number $n$, how can we write it as a sum
of natural numbers (larger than one)?

\smallskip

This is arguably a classical problem, and in order to recall its answer, let us take a simple example: for $n=6$, the possible decompositions can be written as  $\{6;2+2+2;2+4;4+2;3+3\}$. Furthermore, notice that the decompositions of 6 can be obtained by adding $+2$ to the decompositions of 4 or adding $1$ to the last number of the decompositions of 5. Extrapolating to a general integer $n$, it is easily seen that the number of decompositions can be expressed as $u_{n-1}$, where $(u_n)_{n\ge 1}$ stands for the Fibonacci sequence. We have thus found a number of decompositions of the form
$$
N_n=5^{-1/2} \left(\rho_2^{n-1}-\rho_1^{n-1}\right),
$$
where the quantities $\rho_1,\rho_2$ appear in formula (\ref{des2m}). Moreover, the number of terms in $T_2$ is given by $N_{2m}-1$, the $-1$ part corresponding to the term $T_1$.

\smallskip

Putting together this expression with (\ref{pellema02}) and the result of Step 3, our claim (\ref{des2m}) is now easily obtained.

\smallskip

\noindent
\textit{Step 6: Proof of (\ref{des2mmu}).}
The proof of (\ref{des2mmu}) follows the same arguments as for (\ref{des2m}). We briefly sketch the main difference between these two proofs, lying in the analysis of the term $U_{n_1,\ldots,n_s}$. Indeed, since we are now dealing with an odd power $2m+1$, the equivalent of (\ref{pellema02}) is an upper bound of the form
\begin{equation}\label{ig1}
\frac1{\varepsilon}\int_{0}^{1}\int_{0}^{1}|f(y_1)|f^2(y_2)I_{\{|y_1-y_2|<\varepsilon^2\}}dy_1dy_2 \!
\int_{[0,1]^{m-1}} \prod_{j=1}^{m-1} f^2(x_j)
I_{\{x_1\geq \cdots\geq
x_{m-1}\}}dx_1\cdots dx_{m-1}.
\end{equation}
Furthermore, applying H\"older's inequality twice, we obtain
\begin{align*}
&\frac1{\varepsilon}\int_{0}^{1}\int_{0}^{1}|f(y_1)|f^2(y_2)I_{\{|y_1-y_2|<\varepsilon^2\}}dy_1dy_2
=\frac1{\varepsilon}\int_{0}^{1}f^2(y_2)\int_{0\vee(y_2-\varepsilon^2)}^{1\wedge(y_2+\varepsilon^2)}|f(y_1)|
dy_1dy_2\\
\leq&\left(\int_0^1\int_0^1f^2(y_1)f^2(y_2)I_{\{|y_1-y_2|<\varepsilon^2\}}dy_1dy_2\right)^{\frac12}\|f\|_{L^2}=
(\phi_f^{\varepsilon})^{\frac12}\|f\|_{L^2},
\end{align*}
and thus we can bound (\ref{ig1}) by
$\frac1{(m-1)!}(\phi_f^{\varepsilon})^{\frac12}\|f\|_{L^2}^{2m-1}$,
which ends the proof.

\end{proof}

Our next technical lemma compares the moments of a Wiener type integral with respect to $\theta^{\ep}$ and with respect to the white noise.

\begin{lema}\label{lm-control3}
Let $m\in\N$,  $f\in C^\al([0,1])$, $i \in \{1,2\}$, $\varepsilon >0$ and for $m\ge 1$, set
\begin{equation*}
J_m=
\left|\frac1{(2m)!}E\left[\left(\int_0^1f(r)\theta^{i,\varepsilon}(r)dr
\right)^{2m}\right]-\frac1{2^mm!}\int_{[0,1]^m }f^2(s_1)\cdots
f^2(s_m) ds_1\cdots ds_m  \right|.
\end{equation*}
 Then

\smallskip

\noindent
(1)
We have
$
J_1
\leq\varepsilon^{2\alpha}\|f\|_{L^2}\|f\|_{\alpha}.
$

\smallskip

\noindent
(2)
For any  $m > 1$, the following inequality holds true, where we recall that $\rho_1,\rho_2$ have been defined just before Lemma \ref{lema01}:
\begin{multline*}
J_m
 \leq \frac1{(m-1)!}\varepsilon^{2\alpha}\|f\|_{\alpha}\|f\|_{L^2}^{2m-1}  \\
 +\frac{k_{\eta}^{2m}}{\sqrt{5}(m-2)!}
{\left(\rho_1^{2m-1}-\rho_2^{2m-1}\right)}
(\phi_f^{\varepsilon})^{\frac12}\|f\|_{L^2}^{2m-2}
+\frac{1}{(m-2)!}\|f\|_{L^2}^{2(m-2)}\phi_f^{\varepsilon}.
\end{multline*}
\end{lema}

\begin{proof}
We divide again this proof into several steps.

\smallskip

\noindent
\textit{Step 1: Variance estimates.}
We prove here the first of our assertions: Notice that
$$\frac12\int_0^1f^2(s_1)ds_1=\frac1{2\varepsilon^2}\sum_{k=1}^{n(\varepsilon)}\int_{(k-1)\varepsilon^2}^{k\varepsilon^2}\int_{(k-1)\varepsilon^2}^{k\varepsilon^2}f^2(s_1)ds_2ds_1.$$
On the other hand
$$
\frac12E\left[\left(\int_0^1f(r)\theta^{i,\varepsilon}(r)dr\right)^2\right]=
\frac1{2\varepsilon^2}
\sum_{k=1}^{n(\varepsilon)}\int_{(k-1)\varepsilon^2}^{k\varepsilon^2}\int_{(k-1)\varepsilon^2}^{k\varepsilon^2}f(r_1)f(r_2)dr_2dr_1.
$$
We thus get
\begin{eqnarray*}
J_1&=&\left|\frac{1}{2\varepsilon^2}
\sum_{k=1}^{n(\varepsilon)}\int_{(k-1)\varepsilon^2}^{k\varepsilon^2}\int_{(k-1)\varepsilon^2}^{k\varepsilon^2}f(r_1)(f(r_2)-f(r_1))dr_2dr_1\right|\\
&=&\left|\frac{1}{2\varepsilon^2}
\int_{0}^{1}\int_{0}^1f(r_1)(f(r_2)-f(r_1))\left(\sum_{k=1}^{+\infty}I_{[k-1,k)^2}\left(\frac{r_1}{\varepsilon^2},\frac{r_2}{\varepsilon^2}\right)\right)dr_2dr_1\right|,
\end{eqnarray*}
and hence this quantity can be bounded as follows:
\begin{eqnarray*}
J_1&\le&\frac{1}{2\varepsilon^2}
\int_{0}^{1}\int_{0}^1|f(r_1)||(f(r_2)-f(r_1))|I_{\{|r_1-r_2|<\varepsilon^2\}}dr_2dr_1\\
&\leq&\frac{1}{2\varepsilon^2}
\int_{0}^{1}|f(r_1)|\|f\|_{\alpha}\int_{0}^1|r_2-r_1|^{\alpha}I_{\{|r_1-r_2|<\varepsilon^2\}}dr_2dr_1\\
&\leq&\varepsilon^{2\alpha}\|f\|_{L^2}\|f\|_{\alpha},
\end{eqnarray*}
which is the first claim of our lemma.

\smallskip

\noindent
\textit{Step 2: decomposition for higher moments:}
We can follow exactly the computations of Lemma \ref{lema01}, Step 1, in order to get
$$
\frac1{(2m)!}E\left[\left(\int_0^1f(r)\theta^{i,\varepsilon}(r)dr\right)^{2m}\right]
=\tilde T_{m}^{1}+\tilde T_{m}^{2},
$$
with $\tilde T_m^j=\frac{T_m^j}{(2m)!}$ for $j=1,2$. Furthermore,
the term $\tilde T_{m}^{2}$ can be bounded as in Lemma \ref{lema01},
and we obtain
\begin{equation}\label{primerf}
|\tilde T_{m}^{2}|\leq
\frac{k_{\eta}^{2m}}{\sqrt{5}(m-2)!} {\left(\rho_2^{2m-1}-\rho_1^{2m-1}\right)}
(\phi_f^{\varepsilon})^{\frac12}\|f\|_{L^2}^{2m-2}.\end{equation}

\smallskip

\noindent
\textit{Step 3: Study of $\tilde T_{m}^{1}$:}
We analyze $\tilde T_{m}^{1}$ in a slightly different way as in Lemma \ref{lema01}. Namely, we first write
\begin{multline}\label{primert}
\tilde T_{m}^{1}=\frac1{2^m\varepsilon^{2m}}\sum_{\footnotesize{\begin{array}{c}k_1,\dots,k_{m}=1\\k_1>\cdots>k_m\end{array}}}^{n(\varepsilon)}\int_{(k_1-1)\varepsilon^2}^{k_1\varepsilon^2}\int_{(k_1-1)\varepsilon^2}^{k_1\varepsilon^2}\cdots
\int_{(k_{m}-1)\varepsilon^2}^{k_{m}\varepsilon^2}\int_{(k_{m}-1)\varepsilon^2}^{k_{m}\varepsilon^2}
f(r_1)\cdots f(r_{2m})\\
 \times I_{\{\{r_1,r_2\}\geq \cdots\geq
\{r_{2m-1},r_{2m}\}\}}dr_1\cdots dr_{2m},
\end{multline}
where we have written $\{a,b\}\geq\{c,d\}$ for $a\wedge b\geq c\vee d$. We will now compare this quantity with another expression of the same type, called $\hat T_{m}^{1}$ and defined by
\begin{equation*}
\hat T_{m}^{1}=\frac1{2^mm!}\int_{[0,1]^m }f^2(s_1)\cdots f^2(s_m) ds_1\cdots
ds_m.
\end{equation*}

\smallskip

Let us thus write $\hat T_{m}^{1}$ as
\begin{eqnarray*}
\hat T_{m}^{1}&=&\frac1{2^m}\int_{[0,1]^m }\prod_{j=1}^{m}f^2(s_j) \,
I_{\{s_1\geq\cdots\geq s_m\}} \, ds_1\cdots
ds_m\\
&=& \frac1{2^m}
\sum_{\footnotesize{\begin{array}{c}k_1,\dots,k_{m}=1\\k_1\geq\cdots
\geq
k_{m}\end{array}}}^{n(\varepsilon)}\int_{(k_1-1)\varepsilon^2}^{k_1\varepsilon^2}\cdots
\int_{(k_{m}-1)\varepsilon^2}^{k_{m}\varepsilon^2} \prod_{j=1}^{m}f^2(s_j) \,  I_{\{s_1\geq\cdots\geq s_m\}} \, ds_1\cdots
ds_m\\
&=& \frac1{2^m}
\sum_{\footnotesize{\begin{array}{c}k_1,\dots,k_{m}=1\\k_1>\cdots>k_m\end{array}}}^{n(\varepsilon)}
\int_{(k_1-1)\varepsilon^2}^{k_1\varepsilon^2}\cdots
\int_{(k_{m}-1)\varepsilon^2}^{k_{m}\varepsilon^2} \prod_{j=1}^{m}f^2(s_j) \,  I_{\{s_1\geq\cdots\geq s_m\}} \, ds_1\cdots ds_m +
\tilde T_{m}^{3},
\end{eqnarray*}
where $\tilde T_{m}^{3}$ represents the part of the sum taken over the indices $k_1,\ldots,k_m$ such that there exist  $l$ satisfying $k_l=k_{l+1}$. However, this latter term can be bounded as in (\ref{pellema02}), yielding
\begin{equation}\label{segonf}
|\tilde T_{m}^{3}|\leq \frac{m-1}{2^m} \frac{1}{(m-2)!}\|f\|_{L^2}^{2(m-2)}\phi_f^{\varepsilon}\leq \frac{1}{2} \frac{1}{(m-2)!}\|f\|_{L^2}^{2(m-2)}\phi_f^{\varepsilon}.
\end{equation}

\smallskip

\noindent
\textit{Step 4: Conclusion.}
Putting together the decompositions we have obtained so far, we end up with
\begin{align*}
J_m & \leq |\tilde T_{m}^{2}|+|\tilde T_{m}^{3}| + \Bigg|
\frac1{2^m\varepsilon^{2m}}\sum_{\footnotesize{\begin{array}{c}k_1,\dots,k_{m}=1\\k_1>\cdots>k_m\end{array}}}^{n(\varepsilon)}\int_{(k_1-1)\varepsilon^2}^{k_1\varepsilon^2}\int_{(k_1-1)\varepsilon^2}^{k_1\varepsilon^2}\cdots
\int_{(k_{m}-1)\varepsilon^2}^{k_{m}\varepsilon^2}\int_{(k_{m}-1)\varepsilon^2}^{k_{m}\varepsilon^2}
\\& \qquad \left[f(r_1)\cdots f(r_{2m})-
f^2(r_1)f^2(r_3)\cdots f^2(r_{2m-1})\right]I_{\{\{r_1,r_2\}\geq
\cdots\geq \{r_{2m-1},r_{2m}\}\}}dr_1\cdots
dr_{2m}\Bigg|\\
& \leq |\tilde T_{m}^{2}|+|\tilde T_{m}^{3}| + \frac1{2^m m!
\varepsilon^{2m}}\int_{[0,1]^{2m}}\left|f(r_1)\cdots f(r_{2m})-
f^2(r_1)f^2(r_3)\cdots f^2(r_{2m-1})\right|\\& \hspace{7cm} \times
I_{\{|r_1-r_2|<\varepsilon^2\}}\cdots I_{
\{|r_{2m-1}-r_{2m}|<\varepsilon^2\}} \, dr_1\cdots dr_{2m}.
\end{align*}
Invoking our estimates (\ref{primerf}) and (\ref{segonf}) on $\tilde T_{m}^{2}$ and $\tilde T_{m}^{3}$, our bound on $J_m$ easily reduced to check that
\begin{multline*}
\frac1{2^mm!\varepsilon^{2m}}\int_{[0,1]^{2m}}\left|f(r_1)\cdots
f(r_{2m})- f^2(r_1)f^2(r_3)\cdots
f^2(r_{2m-1})\right|\\
\times I_{\{|r_1-r_2|<\varepsilon^2\}}\cdots I_{
\{|r_{2m-1}-r_{2m}|<\varepsilon^2\}}dr_1\cdots
dr_{2m} \leq \frac1{(m-1)!}\varepsilon^{2\alpha}\|f\|_{\alpha}\|f\|_{L^2}^{2m-1}.
\end{multline*}
The latter inequality can now be obtained from the decomposition
\begin{eqnarray*}
&&\left|f(r_1)\cdots f(r_{2m})- f^2(r_1)f^2(r_3)\cdots
f^2(r_{2m-1})\right|\\
&&\quad =\left|f(r_1)(f(r_2)-f(r_1))f(r_3)f(r_4)\cdots f(r_{2m})\right.\\
&&\qquad +f^2(r_1)f(r_3)(f(r_4)-f(r_3))f(r_5)f(r_6)\cdots f(r_{2m})\\
&&\qquad +\cdots\\
&&\qquad \left.+f^2(r_1)f^2(r_3)...f^2(r_{2m-3})f(r_{2m-1})(f(r_{2m})-f(r_{2m-1}))\right|,
\end{eqnarray*}
the inequalities
\begin{eqnarray*}
 \frac{1}{2\varepsilon^2}\int_0^1\int_0^1f^2(r_1)I_{|r_1-r_2|<\varepsilon^2}dr_2 dr_1&\leq&\|f\|_{L^2}^2,\\
\frac{1}{2\varepsilon^2}\int_0^1\int_0^1f(r_1)f(r_2)I_{|r_1-r_2|<\varepsilon^2}dr_2 dr_1&\leq&\|f\|_{L^2}^2,
\end{eqnarray*}
and from the estimate we have already obtained for $J_1$. This finishes the proof.

\end{proof}

\smallskip

Finally, the characteristic function of a Wiener type integral of the form $\int_0^1 f(r)\theta^{k,\varepsilon}(r)dr$ can be compared to its expected limit $\int_0^1 f(r)dW^k_r$ in the following way:

\begin{lema}\label{lm-control2}
Let $f\in{\mathcal{C}}^\alpha([0,1])$ for a certain $\alpha\in(0,1)$,
$k\in\{1,\ldots,d\}$ and $\varepsilon>0$. For any $u\in\R$, we have:
\begin{eqnarray*}
&&\left|E\big[{\rm e}^{iu\int_0^1
f(r)\theta^{k,\varepsilon}(r)dr}\big]
-E\big[{\rm e}^{iu\int_0^1 f(r)dW^k_r}\big]\right|\\
&& \quad \leq 4 (1/5)^{1/2}
u^3(\phi_f^{\varepsilon})^{\frac12}\|f\|_{L^2}
k_{\eta}^3\exp(4 u^2 k_{\eta}^2 \|f\|_{L^2}^2) +u^2\varepsilon^{2\alpha}\|f\|_{\alpha}\|f\|_{L^2}\exp(u^2\|f\|_{L^2}^2)\\
&&\qquad +8 \, 5^{-1/2}
u^4(\phi_f^{\varepsilon})^{\frac12}\|f\|_{L^2}^2 k_{\eta}^2\exp(4u^2
k_{\eta}^2 \|f\|_{L^2}^2)+ (1/2)
u^4\phi_f^{\varepsilon}\exp(u^2\|f\|_{L^2}^2).
\end{eqnarray*}

\end{lema}

\begin{proof}
Let us control first the imaginary part of the difference. Using lemma \ref{lema01}, and invoking the fact that the odd moments of a Gaussian random variable are null, we get
\begin{eqnarray*}
&&\left|{\rm Im}\left(E\big[{\rm e}^{iu\int_0^1
f(r)\theta^{k,\varepsilon}(r)dr}\big] -E\big[{\rm e}^{iu\int_0^1
f(r)dW^k_r}\big]\right)\right|\\
& & \quad \leq \sum_{m=1}^\infty \frac{|u|^{2m+1}}{(2m+1)!}
\left|E\left[\left( \int_0^1 f(r)\theta^{k,\varepsilon}(r)dr
\right)^{2m+1}
\right]\right|\\
& & \quad \leq \sum_{m=1}^\infty
\frac{(1/5)^{1/2}|u|^{2m+1}}{(m-1)!}k_{\eta}^{2m+1}
 \lp \rho_2^{2m}-\rho_1^{2m}\rp
\, (\phi_f^{\varepsilon})^{\frac12}\|f\|_{L^2}^{2m-1}\\
& & \quad \leq 4 (1/5)^{1/2}
u^3(\phi_f^{\varepsilon})^{\frac12}\|f\|_{L^2}
k_{\eta}^3\sum_{m=1}^\infty\frac1{(m-1)!}\left(4k_{\eta}^2\|f\|_{L^2}^2u^2\right)^{m-1}
\\
 & & \quad \leq 4(1/5)^{1/2} u^3(\phi_f^{\varepsilon})^{\frac12}\|f\|_{L^2}
k_{\eta}^3\exp(4 u^2 k_{\eta}^2 \|f\|_{L^2}^2) .
\end{eqnarray*}

\smallskip

In order to control the real part of the difference, we will use Lemma \ref{lm-control3}. This yields:
\begin{align*}
&\left|{\rm Re}\big(E\big[{\rm e}^{iu\int_0^1
f(r)\theta^{k,\varepsilon}(r)dr}\big] -E\big[{\rm e}^{iu\int_0^1
f(r)dW^k_r}\big]\big)\right|\\
\le&  \sum_{m=1}^{+\infty}u^{2m}\Bigg[
\frac1{(m-1)!}\varepsilon^{2\alpha}\|f\|_{\alpha}\|f\|_{L^2}^{2m-1} \\
&
\hspace{3cm}+u^4(\phi_f^{\varepsilon})^{\frac12}\frac8{\sqrt{5}}\|f\|_{L^2}^2
k_{\eta}^2\sum_{m=2}^{+\infty}\frac1{(m-2)!}\left(4k_{\eta}^2u^2\|f\|_{L^2}^2\right)^{m-2}\\
& \hspace{3cm}+ \frac12 u^4\phi_f^{\varepsilon}\sum_{m=2}^{+\infty}\frac1{(m-2)!}\left(u^2\|f\|_{L^2}^2\right)^{m-2}\Bigg].
\end{align*}
The latter quantity can be bounded by
\begin{multline*}
u^2\varepsilon^{2\alpha}\|f\|_{\alpha}\|f\|_{L^2}\exp(u^2\|f\|_{L^2}^2)
+\frac8{\sqrt{5}}u^4(\phi_f^{\varepsilon})^{\frac12}\|f\|_{L^2}^2
k_{\eta}^2\exp(4u^2 k_{\eta}^2 \|f\|_{L^2}^2)\\
+ \frac12 u^4\phi_f^{\varepsilon}\exp(u^2\|f\|_{L^2}^2),
\end{multline*}
which ends the proof.

\end{proof}


\end{document}